\documentclass[12pt]{article}
\usepackage{amsmath}

\usepackage[utf8]{inputenc}
\usepackage[english]{babel}
\usepackage{amsmath, amsfonts, amsthm, amssymb, setspace, enumitem, tikz-cd, bbm, mathtools, graphicx,scalerel, mathrsfs, titling, hyperref}
\usepackage{graphicx}
\graphicspath{ {./images/} }

\DeclareMathOperator{\Hilb}{\mathcal H}

\newcommand{\tensor}{\otimes}

\newtheoremstyle{break}
  {\topsep}{\topsep}%
  {\bfseries}{}%
  {\newline}{}%
\theoremstyle{break}
\newtheorem{theorem}{Theorem}[section]
\newtheorem{theorem-break}[theorem]{Theorem}
\newtheorem{lemma}[theorem]{Lemma}
\newtheorem{corollary}[theorem]{Corollary}
\newtheorem{proposition}[theorem]{Proposition}

\theoremstyle{definition}
\newtheorem{definition}[theorem]{Definition}
\newtheorem{example}[theorem]{Example}

\title{The Quasicentral Modulus Associated with a Class of Nonself-similar Fractals}
\author{R. Alexander Glickfield}
\begin{document}

\begin{titlingpage}
    \maketitle
    \begin{abstract}
We discuss an extension to Voiculescu's formula for the quasicentral modulus of a tuple of commuting, self-adjoint operators with spectral measure absolutely continuous with respect to a generalized Hausdorff measure. These Hausdorff measures are  defined by gauge functions which are not power functions and are supported on nonself-similar fractals.
    \end{abstract}
\end{titlingpage}

\section{Introduction}

\hspace{12pt} The possibility of diagonalizing a self-adjoint operator by a small compact perturbation was first observed by Weyl \cite{von1909beschrankte}. He showed that any self-adjoint operator is the sum of a diagonal operator (in some orthonormal basis) with a compact self-adjoint operator. Von Neumann later improved this result by considering Hilbert-Schmidt perturbations \cite{von1935charakterisierung}. The result was further extended by Kuroda who showed that it is possible to diagonalize a self-adjoint operator by a perturbation contained in any normed ideal other than the trace class \cite{kuroda1958theorem}. Later, Berg \cite{berg1971extension} extended Weyl's compact perturbation result to normal operators. It was not known whether compact operators can be replaced by Hilbert-Schmidt operators in Berg's result until Voiculescu's work \cite{voiculescu1979some}, a corollary of which being an affirmative answer to this question. Voiculescu discovered that the simultaneous diagonalization modulo a normed ideal of an $n$-tuple of commuting self-adjoint operators is equivalent to the existence of a quasicentral (as measured by the norm in the ideal) approximate unit of finite rank operators. A natural question is to find the obstruction ideals, that is, the normed ideals $\mathfrak{S}_\Phi$ for which a given $n$-tuple of commuting self-adjoint operators \textit{cannot} be expressed as the sum of a diagonal $n$-tuple with an $n$-tuple of operators in $\mathfrak{S}_{\Phi}$.

The first example of an obstruction ideal was found by Kato \cite{kato1957perturbation} and Rosenblum \cite{rosenblum1957perturbation} who proved that the absolutely continuous part of the spectral measure of a self-adjoint operator is invariant under trace-class perturbations. (When a self-adjoint operator has only singular spectrum, Carey and Pincus \cite{carey1974perturbation} proved that the trace class is \textit{not} an obstruction ideal). Voiculescu found in \cite{voiculescu1979some} a formula for the quasicentral modulus of $n$-tuples of operators relative to the Lorentz-type ideals $\mathfrak{S}_n^-$. This modulus is zero precisely when diagonalization modulo the ideal is possible. The formula has also been generalized to the Lorentz ideals $\mathfrak{S}_p^-$ in the case of certain $n$-tuples of operators with self-similar spectra of noninteger Hausdorff dimension $p$ (see \cite{voiculescu2021formula} and also \cite{ikeda2023quasicentral} for more recent work).

The goal of this paper is two-fold. First, given an $n$-tuple of operators whose spectral measure is absolutely continuous with respect to a generalized Hausdorff measure generated by some gauge function, we find the largest obstruction ideal. Second, we extend the integral formula for the quasicentral modulus to these $n$-tuples. To this end, we discuss necessary preliminaries in Section 2. In Section 3, we construct the maximal obstruction ideals for the relevant multiplication operators. In Section 4, we prove an ampliation homogeneity result which is central to the extension of Voiculescu's formula. In Section 5, we state and prove the formula for the quasicentral modulus in these cases.

I would like to thank the reviewer for their careful reading of the manuscript and for providing valuable suggestions that have substantially improved the paper.

\section{Notation and Preliminaries}

\subsection{Operator Preliminaries}

We denote by $B(\Hilb)$ the algebra of bounded linear operators on a separable, infinite dimensional, complex Hilbert space $\Hilb$. We write $c_{00}$ for the vector space of real sequences $\{\xi_j\}_{j=1}^{\infty}$ that have only finitely many nonzero terms, and let $\mathcal {SN}$ denote the set of norms $\Phi: c_{00}\rightarrow [0,\infty)$ that satisfy $\Phi(\xi)=\Phi(\xi^*)$, where $\xi^*$ is the nonincreasing rearrangement of $\{|\xi_j|\}_{j=1}^\infty$. These functions $\Phi$ are known as \textit{symmetric norming functions}. It should be noted that one may normalize an element $\Phi\in \mathcal{SN}$ by replacing it with $\Phi/\Phi(1,0,0,\dots)$, but this is not required for the results in this paper.

Given $\Phi\in \mathcal{SN}$ and a finite rank operator $F$ with singular values $(s_j)_{j=1}^\infty$, we define $|F|_{\Phi}:=\Phi((s_j)_{j=1}^\infty)$. For arbitrary $T\in B(\Hilb)$, we set $|T|_{\Phi} = \sup_{P\in\mathcal{P}}|PT|$, where $\mathcal P$ denotes the directed set of finite rank projections on $\Hilb$. The set $\mathfrak S_\Phi = \{T\in B(\Hilb) : |T|_\Phi <+\infty\}$ is an ideal and $(\mathfrak{S}_\Phi,|\cdot|_\Phi)$ is a Banach space. For all the norms we consider, the ideal of finite rank operators is dense in the Banach space $\mathfrak{S}_{\Phi}$.

We denote by $\mathcal {R}_1^+$ the directed set of positive, finite rank contractions, where $A\geq B$ if $A-B$ is a positive operator. For $A,T\in B(\Hilb)$, we recall (see \cite[page 6]{voiculescu1979some}) that the quasicentral modulus $k_\Phi(T)$ of an operator $T$ is defined as
\[k_\Phi (T) := \liminf_{A\in \mathcal{R}_1^+} |[A,T]|_\Phi,\]
\noindent where $[A,T] = AT-TA$. Consider now $n$-tuples $\tau^{(j)}=(T_1^{(j)},\dots,T_n^{(j)})$ with $T_k^{(j)}\in (B(\Hilb))^n$, where $j=0,1,2$, and let $A,B\in B(\Hilb)$. We define the following operations:
\begin{align*}
A\tau^{(0)}B&:=(AT_1^{(0)}B,\dots,AT_n^{(0)}B),\\
\tau^{(1)}+\tau^{(2)}&:=(T_1^{(1)}+T_1^{(2)},\dots,T_n^{(1)}+T_n^{(2)}).
\end{align*}
We also set $||\tau||:= \max_{k}||T_k||$, and $|\tau|_\Phi:= \max_{k}|T_k|_\Phi$. The definition of the quasicentral modulus extends to $n$-tuples of operators as follows:
\[k_\Phi (\tau) := \liminf_{A\in \mathcal{R}_1^+} |[A,\tau]|_\Phi.\]
 A commuting $n$-tuple of self-adjoint operators $\tau$ is said to be \textit{diagonalizable moduluo} $\mathfrak{S}_\Phi$ if there exists a diagonal $n$-tuple $\Delta\in (B(\Hilb))^n$ so that $\tau - \Delta \in (\mathfrak{S}_\Phi)^n$. By \cite[Corollary 2.6]{voiculescu1979some}, $k_\Phi(\tau)>0$ if and only if $\tau$ is \textit{not} diagonalizable modulo $\mathfrak{S}_\Phi$. We say that $\mathfrak{S}_\Phi$ is an \textit{obstruction ideal} for $\tau$ when $k_\Phi(\tau) > 0$, and $\mathfrak{S}_\Phi$ is a \textit{diagonalization ideal} for $\tau$ otherwise. As shown in \cite{bercovici1989analogue}, any obstruction ideal of $\tau$ is contained in another obstruction ideal of the form $\mathfrak{S}_{\Phi_\pi}$ where $\pi = (\pi_k)_{k=1}^\infty\subset \mathbb R$ is a non-increasing real sequence such that $\pi_k\rightarrow 0$, $\sum_{k=1}^\infty \pi_k = +\infty$, and
\[\Phi_\pi(\xi):=\sum_{k=1}^\infty\pi_k\xi_k^*.\]
We also write $|\cdot|_\pi$ and $k_\pi(\cdot)$ for $|\cdot|_{\Phi_\pi}$ and $k_{\Phi_\pi}(\cdot)$, respectively. The results of interest found in \cite{david1990s} require the following regularity condition on the sequences $\pi$ defining obstruction ideals $\mathfrak{S}_{\pi}$.

\begin{definition}
A sequence $\pi$ as above is said to be \textit{regular} if there exists a constant $\alpha>0$ such that
\[\sum_{k=1}^m \pi_k \leq \alpha m \pi_m\] 
for every $m\in \mathbb N$.
\end{definition}

\subsection{Hausdorff Measures}

Fix a strictly increasing function $f:[0,\infty)\rightarrow[0,\infty)$ such that $f((0,+\infty))\subset (0,+\infty)$ and $\lim_{t\rightarrow 0} f(t) = 0$. Given an arbitrary subset $A\subset\mathbb {R}^n$ and $r>0$, a collection of open balls $\{B(x_j,r_j)\}_{j\in \mathbb N}$ is said to be an $r$\textit{-cover} of $A$ if $r_j<r$ for every $j\in \mathbb N$ and $A\subset \bigcup_{j=1}^\infty B(x_j,r_j)$. The \textit{outer Hausdorff measure} of $A$ using the \textit{gauge function} $f$ is defined as

\[H_{f}^*(A):=\lim_{r\rightarrow 0}\inf\bigg\{\sum_{j=1}^\infty f(r_j) : \{B(x_j,r_j)\} \text{ is an }r\text{-cover of A}\bigg\}\] 

\noindent where the infimum is taken over the set of $r$-covers of $A$. The restriction of $H_{f}^*$ to the $\sigma$-algebra of Borel sets in $\mathbb R^n$ is a measure denoted $H_f$.
We often assume the following regularity conditions on the gauge functions.
\begin{definition}
Given $s> 1$, we say that a gauge function $f:[0,\infty)\rightarrow[0,\infty)$ has property $(R_s)$ if:
\begin{enumerate}
\item $f$ is convex,
\item $f(1) = 1$,
\item $f\in C^1([0,\infty))\cap C^2((0,\infty))$,
\item $f'(0) = 0$,
\item $\lim_{x\rightarrow 0} f(ax)/f(x) = a^s\quad \text{for every } a > 0.$
\end{enumerate}
\end{definition}
Note that the inverse with respect to composition of a gauge function with property $(R_s)$ exists, but is concave and hence does not have property $(R_s)$. Functions satisfying the last condition of Definition 2.2 are said to be \textit{regularly varying} at 0 with index $s$. We make note of a few useful facts about regularly varying functions \cite[Section 1.4.1]{bingham1989regular}.
\begin{lemma}
    
If $f$ is a gauge function with property $(R_s)$, then 

\[\lim_{x\rightarrow 0} \frac{f^{-1}(x)}{f^{-1}(ax)} = \frac{1}{a^{1/s}}\]

\noindent where $f^{-1}$ is the compositional inverse of f.
\end{lemma}
\begin{proof}
Theorem 1.5.12 of \cite{bingham1989regular} implies $f^{-1}$ is regularly varying at 0 with index $1/s$, so 
\[\lim_{x\rightarrow 0} \frac{f^{-1}(ax)}{f^{-1}(x)} = a^{1/s}.\]
The stated equality follows from (left) composing the quotient $f^{-1}(ax)/f^{-1}(x)$ with the function $x\rightarrow 1/x$ and letting $x$ tend to 0.
\end{proof}
\begin{lemma}
If $f$ is a regularly varying at $0$ with index $s$, then
\[
f(x) = x^s L(x),
\]
where $L$ is a regularly varying function at $0$ with index $0$.
\end{lemma}
Such functions $L$ are called \textit{slowly varying} in the literature. We apply this representation to both $f$ and its inverse $f^{-1}$, and denote the corresponding slowly varying functions by $L(x)$ and $P(x)$, respectively. The next Lemma is a convenient specialization of the Potter bounds found in \cite[Theorem 1.5.6]{bingham1989regular}.
\begin{lemma}
Let $L$ be a positive slowly varying function at $0$. For every $\delta>0$, there exists $x_0>0$ such that for every $0<a<1$ and $0<x\leq x_0$,
\[\frac{L(ax)}{L(x)} \leq a^{-\delta}\]
\end{lemma}

Finally, we introduce a normalized version of the Karamata Representation Theorem \cite[Theorem 1.3.3]{bingham1989regular}.

\begin{lemma}
A function $L$ is slowly varying at 0 if and only if it may be written in the form
\[L(x) = C\exp\bigg({\int_x^{x_0}}\frac{\varepsilon(u)}{u}\  du\bigg)\]
Where $C>0$ is a constant, $x_0>0$, and $\varepsilon(u)$ is a measurable function with $\lim_{x\rightarrow 0}\varepsilon(x) = 0$.
\end{lemma}

\subsection{Construction of Relevant Fractals}

We define a fractal as a compact set in $\mathbb R^n$ that has non-integer Hausdorff dimension. Many of our results rely on the spectrum of a tuple of operators being contained in a fractal $K$ such that $0<H_f(K)<+\infty$. We also often require that the fractal satisfies the regularity condition $Cf(r) \leq H_f(B(x,r)\cap K) \leq Df(r)$ for every $x\in K$ and some constants $C,D>0$ independent of $x$. This subsection will provide the construction of such a fractal for the measure $H_f$.

We begin by discussing the construction of certain symmetric generalized Cantor sets. A more detailed discussion can be found in Section 2 of \cite{hatano2004hausdorff}. Given a gauge function $f:[0,\infty)\rightarrow[0,\infty)$ with property $(R_s)$, let $n=\lfloor s\rfloor+1$. This selection of $n$ ensures that $f(x)/x^n$ is eventually decreasing and thus allows the use Theorem 1 of \cite{hatano2004hausdorff}. By Property $(R_s)$, one can choose a positive integer $M$ such that $f^{-1}\Bigl(\frac{1}{2^{(m-1)nM}}\Bigr) > 2 f^{-1}\Bigl(\frac{1}{2^{mnM}}\Bigr)$ for all $m > M$. This ensures that the subsequent construction is well-defined. We define the fractal inductively: Beginning with the unit interval, remove an open interval of length $\eta_{1} = 1-2f^{-1}(1/2^{nM})$ so that the two remaining closed intervals $I_1^1, I_2^1$ are of equal length $\lambda_{1} = f^{-1}(1/2^{nM})$. For the $m$-th step, remove from each closed interval $I_j^{m-1}$ an open interval of length $\eta_m = f^{-1}(1/2^{(m-1)nM})-2f^{-1}(1/2^{mnM})$ so that the resulting collection of $2^m$ intervals $I_j^{m}$ have equal length $\lambda_m = f^{-1}(1/2^{nmM})$. It should be noted that Lemma 2.3 implies that $\sum_{m=1}^\infty2^{m-1}\eta_m = 1$, since $\lim_{m\rightarrow\infty} 2^mf^{-1}(1/2^{nmM}) = 2^{m(1-nM/s)}P(1/2^{nmM}) = 0$. The last equality is a consequence of Proposition 1.3.6 of \cite{bingham1989regular}. Since the examples of interest allow a choice of $M=1$, we will assume $M=1$ for the remainder of the paper for ease of notation. We focus on the symmetric generalized Cantor set $C_{f}\subset \mathbb R^n$ which is the product of $n$ copies of \[\bigcap_{m=1}^{\infty}\bigcup_{j=1}^{2^m} I_{j}^{m}\]
Observe that $C_f$ is a union of $2^n$ identical (up to translation) sets, each of which contains exactly one corner point of the unit $n$-cube. These sets are said to be in the first generation of $C_f$. Similarly, every set from the $L$th generation of $C_f$ is the union of $2^n$ sets from the $(L+1)$st generation. 
Consider the lexicographic ordering of the corner points of the unit $n$-cube. For example, the order of the corner points of the unit square is:
\[(0,0)\leq (0,1)\leq (1,0)\leq (1,1)\]
We write $C_f^{k}$, $k =1,2,\dots 2^n$  for the set in the first generation that contains the $k$th corner point (relative to the lexicographic order). Likewise, observe that sets in the $L$th generation each contain exactly one outside corner point of a set in the $(L-1)$st generation. Thus we may index the sets in the $L$th generation with words of length $L$ whose "letters" are in $\{1,\dots,2^n\}$. We write $C_f^{w}$ with $w\in \prod_{1}^L\{1,\dots,2^n\}$ to identify the $L$th generation sets and denote the length $L$ of $w$ by $|w|$. See Figure 1 for a picture of the second generation for $C_f$ when $n=2$.
Clearly, for any fixed length $L$ we have
\[C_f = \bigcup_{|w|=L} C_f^w\]
It was shown in \cite[Theorem 1]{hatano2004hausdorff} that $0<H_f(C_{f})<+\infty$ and, more generally $\xi\cdot2^{-n|w|}<H_f(C_{f}^w)<2^{-n|w|}$ for some $\xi>0$ dependent only on $n$. It is easy to see that there exist constants $C,D > 0$ independent of $x$ so that $Cf(r)\leq H_f(C_f \cap B(x,r)) \leq Df(r)$ for every $r>0$.

\begin{center}
\includegraphics[scale=0.0525]{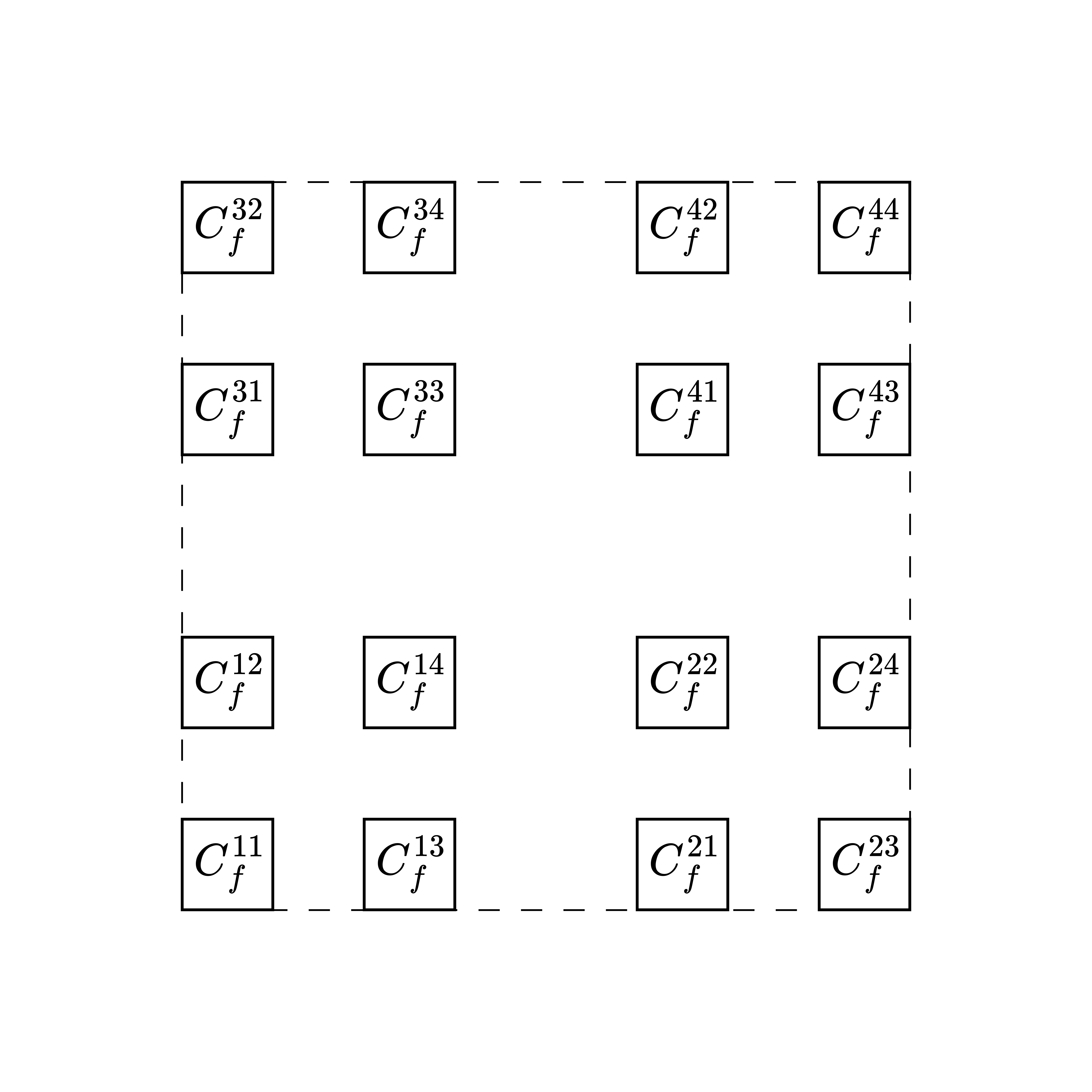}
\end{center}
\begin{center}
\vspace{-1.2cm}\textbf{Figure 1}
\end{center}

\subsection{Preliminary Lemmas}

We begin by recalling a convenient way to decompose a Hilbert space. 

\begin{lemma} Consider a gauge function $f$ with property $(R_s)$ and an $n$-tuple $\tau$ of commuting self-adjoint operators on $\Hilb$ satisfying $0<H_{f}(\sigma(\tau))<+\infty$. Let $E_{\tau}$ be the spectral measure of $\tau$, $\Hilb^f_{ac}$ be the space of vectors $v$ such that $\langle E_\tau(\cdot)v,v\rangle$ is absolutely continuous with respect to $H_{f}$, and $\Hilb^f_{s}$ be the space of vectors $v$ such that $\langle E_\tau(\cdot)v,v\rangle$ is singular with respect to $H_{f}$. The spaces $\Hilb_s^f$ and $\Hilb_{ac}^f$ are closed, reducing for $\tau$, and $\Hilb=\Hilb^f_{ac} \oplus \Hilb^f_{s}$. 
\end{lemma}

In Section 5 we show that the problem of finding an obstruction ideal for $\tau$ amounts to finding an obstruction ideal for multiplication by the variables on $L^2(\sigma(\tau), H_{f})$ whose definition we recall. Given a bounded set $\Omega\subset \mathbb R^n$ and a finite Borel measure $\mu$ on $\Omega$, we denote by $\tau_{\Omega,\mu}$ the multiplication by the variables on $L^2(\Omega, \mu)$, that is
\begin{align}
\tau_{\Omega,\mu} &:= (T_{1,\mu}, T_{2,\mu}, \dots, T_{n,\mu}), \text{ where}\\
(T_{j,\mu} f)(x) &:= x_jg(x), g\in L^2(\Omega, \mu).
\end{align}
The following results are standard and can be found in \cite[Section 1]{voiculescu1979some}.

\begin{lemma} Let $\tau\in (B(\Hilb))^n$ and let $(A_{j})_{j=1}^\infty\subset \mathcal{R}_1^+$ be a sequence that converges to the identity operator in the weak operator topology. Then 
\[k_\Phi(\tau)\leq \liminf_{j} |[A_j,\tau]|_\Phi .\]
\end{lemma}

\begin{lemma}  Let $\tau\in (B(\Hilb))^n$. Then there exists an increasing sequence $(A_j)_{j=1}^\infty\subset \mathcal{R}_1^+$ such that $A_j$ converges to the identity operator in the weak operator topology, and
\[k_\Phi(\tau)= \lim_{j\rightarrow\infty} |[A_j,\tau]|_\Phi .\]\\
Furthermore, the $A_j$ can be chosen so that $\lim ||[A_j,\tau]|| = 0$.

\end{lemma}

\begin{lemma} If $\{\tau^{(j)}\}_{j=1}^\infty\subset B(\Hilb)^n$, $\lambda^{(j)}\in\mathbb C^n$, then:
\begin{enumerate}
\item $\max_{j=1,2}\{k_\Phi(\tau^{(j)})\}\leq k_\Phi(\tau^{(1)}\oplus\tau^{(2)})\leq k_\Phi(\tau^{(1)})+k_\Phi(\tau^{(2)})$,
\item $k_\Phi(\bigoplus_{j=1}^\infty \tau^{(j)}) = \lim_{N\rightarrow\infty} k_\Phi(\bigoplus_{j=1}^N \tau^{(j)})$, and
\item $k_\Phi(\tau^{(1)}\oplus\dots\oplus\tau^{(N)}) = k_\Phi((\tau^{(1)} - \lambda^{(1)} \tensor I)\oplus\dots\oplus(\tau^{(N)}- \lambda^{(N)} \tensor I))$.
\end{enumerate}
\end{lemma}

\section{Obstruction Ideals for $\tau_{\Omega,H_f}$}

In this section we consider a gauge function $f$ with property $(R_s)$, and fractals $\Omega\subset \mathbb R^n$ that satisfy the regularity condition 
\[Cf(r) \leq H_f(B(x,r)\cap \Omega) \leq Df(r)\]
for every $x\in\Omega$, where $C,D$ are independent of $x$. The following lemma is essentially a restatement of results from \cite{voiculescu1979some} and \cite{david1990s}.

\begin{lemma}\label{lem_obstruction} Let $f:[0,+\infty)\rightarrow[0,+\infty)$ be a strictly increasing function with property $(R_s)$, and let $\pi = (\pi_k)_{k=1}^\infty$ be a regular (see Definition 2.1) non-increasing sequence such that $\pi_k\rightarrow 0$ and $\sum_{k=1}^\infty \pi_k = +\infty$. Suppose that

\[0<\inf_{m}m\pi_mf^{-1}(1/m)\leq\sup_{m} m\pi_mf^{-1}(1/m) < +\infty\]

\noindent Then $0<k_\pi(\tau_{\Omega,H_f})<+\infty$.

\begin{proof}

We seek to use Theorem 4.6 from \cite{david1990s} to prove that $0<k_{{\pi}}(\tau_{\Omega,H_f})$. To this end, we verify that there exist positive constants $C_1, C_2$ such that for every $x\in(0,1/2]$ and $\ell\geq1$,\[
f(2x) \leq C_1 f(x)
\quad\text{and}\quad
\sum_{k>\ell} 2^k f(2^{-k}) \le C_2\, 2^\ell f(2^{-\ell}).
\]
For the first inequality, let 
\[s_k := \max\left\{ \frac{f(2x)}{f(x)} : 2^{-k} \leq x \leq 2^{-k+1} \right\}.\]
Since $\lim_{x\rightarrow0}f(2x)/f(x) = 2^s$, the sequence $\{s_k\}_{k=1}^\infty$ is bounded by a constant $C_1$. Therefore, for every $x\in(0,1/2]$, $f(2x)/f(x) \leq C_1$ and hence $f(2x) \leq C_1 f(x)$. For the second inequality, it suffices to show that the following infinite sums are uniformly bounded in $\ell$:
\[\frac{1}{2^\ell}\sum_{k>\ell}2^{k}\bigg(\frac{f(2^{-k})}{f(2^{-\ell})}\bigg).\]
Using Lemma 2.4, we write $f(x) = x^sL(x)$. Thus
\begin{align*}
\frac{1}{2^\ell}\sum_{k>\ell}2^{k}\bigg(\frac{f(2^{-k})}{f(2^{-\ell})}\bigg) &= \frac{1}{2^\ell}\sum_{k>\ell}2^{k}\bigg(\frac{2^{-sk}L(2^{-k})}{2^{-s\ell}L(2^{-\ell})}\bigg)\\
&= \frac{1}{2^\ell}\sum_{m=1}^\infty2^{m(1-s)+\ell}\bigg(\frac{L(2^{-m}\cdot2^{-\ell})}{L(2^{-\ell})}\bigg)\\
&= \sum_{m=1}^\infty2^{m(1-s)}\bigg(\frac{L(2^{-m}\cdot2^{-\ell})}{L(2^{-\ell})}\bigg)\\
\end{align*}
By Lemma 2.5, we may choose $\delta<s-1$ so that for every $\ell>M$
\[\frac{L(2^{-m}\cdot2^{-\ell})}{L(2^{-\ell})}\leq 2^{\delta m}.\]
The tail of the series can now be bounded like so
\[
\sum_{m=M}^\infty2^{m(1-s)}\bigg(\frac{L(2^{-m}\cdot2^{-\ell})}{L(2^{-\ell})}\bigg) \leq \sum_{m=M}^\infty2^{m(1-s+\delta)}.
\]
Since $1-s+\delta < 0$, the right-hand side is a convergent geometric series, completing the proof.

We now show $k_{{\pi}}(\tau_{\Omega,H_f})<+\infty$. The condition $H_f(B(x,r)\cap \Omega)\leq Df(r)$ allows one to cover $\Omega$ by disjoint Borel sets $\{\omega_{\ell,m}\}_{\ell=1}^{m^n}$ with $\text{diam}(\omega_{\ell,m})=f^{-1}(1/{m^n})$ so that $H_{f}(\omega_{\ell,m})\leq D/m^n$ for every positive integer $m$. Consider the sequences of projections
\[P_{m}^{(\ell)}v = \langle v, \chi_{\omega_{\ell,m}} \rangle \frac{\chi_{\omega_{\ell,m}}}{||\chi_{\omega_{\ell,m}}||_2^2}, \quad v\in L^2(\Omega, H_f)\]
and
\[P_m=\sum_{\ell=1}^{m^n} P_m^{(\ell)}\]
Observe that the sequence $P_m$ converges strongly to $I_{L^2(H_f,\Omega)}$. Write $T_j$ for the $j$th component of $\tau_{\Omega,H_f}$. We see that $P_{m}^{(\ell)}T_j$ and $T_j P_{m}^{(\ell)}$ each have rank at most 1, hence $[P_m^{(\ell)}, T_j]$ has rank at most 2. Furthermore, note that $||[P_m,\tau_{\Omega,H_f}]|| \leq 2\text{diam}(\omega_{\ell,m})\leq2\sqrt{n}Df^{-1}(1/{m^n})$. Thus, if $\alpha>0$ such that $\sum_{k=1}^m \pi_k \leq \alpha m \pi_m$ for every $m\in \mathbb N$, we have
\begin{align*}
    k_{\pi}(\tau_{\Omega,H_f})\leq|[P_m,\tau_{\Omega,H_f}]|_{{\pi}} &\leq 2D\sqrt{n}f^{-1}(1/{m^n})\Phi_\pi(\underset{m^n \text{ times}}{\underbrace{1,1,\dots,1}},0,0,\dots)\\
    &= 2D\sqrt{n}f^{-1}(1/{m^n})\sum_{k=1}^{m^{n}}\pi_k\\
    &\leq2\alpha D\sqrt{n}\sup_{m} m\pi_mf^{-1}(1/m) < +\infty .\qedhere
\end{align*}
\end{proof}

\end{lemma}

Before proceeding further, we need to perform an exercise in calculus.

\begin{lemma} Let $f:[0,+\infty)\rightarrow [0,+\infty)$ be a function with property $(R_s)$, and define $h(x):=1/f^{-1}(1/x)$. Furthermore, assume that there exists $t_0>0$ such that for any $t\in[0,t_0]:$ 
\begin{align}
2f(t)f'(t) + tf(t)f''(t) - 2tf'(t)^2 < 0. 
\end{align}
Then there exists $x_0$ such that $h''(x)\leq 0$ for every $x>x_0$.

\begin{proof}

Observe that

\begin{align*}
h''(x) =& \frac{1}{f'(f^{-1}(1/x))}\Bigg[\frac{2}{(xf^{-1}(1/x))^3}\Bigg(\frac{1}{xf'(f^{-1}(1/x))}-f^{-1}(1/x)\Bigg)\Bigg]\\
&+\frac{1}{[xf^{-1}(1/x)]^2}\Bigg[\frac{1}{[f'(f^{-1}(1/x))]^2}\Bigg(\frac{f''(f^{-1}(1/x))}{x^2f'(f^{-1}(1/x))}\Bigg)\Bigg].
\end{align*}\\
Setting $t=f^{-1}(1/x)$, we obtain

\begin{align*}
h''(t) &= \Bigg(\frac{f(t)}{tf'(t)}\Bigg)^3\Big(2f(t)f'(t) + tf(t)f''(t) - 2tf'(t)^2\Big).
\end{align*}
Choosing $x_0=1/f(t_0)$ yields the result.
\end{proof}

\end{lemma}

\begin{corollary}With the notation of Lemma $3.2$, there exists some positive integer $N$ so that $\rho:=\{h'(k+N-1)\}_{k=1}^\infty$ defines a symmetric-norming function $\Phi_\rho$.

\begin{proof}
Set $\rho_k = h'(k+N-1)$. By Lemma 3.2, we know that for some positive integer $N$, the sequence $\rho = (\rho_k)_{k=1}^\infty$ is nonincreasing. Thus it suffices to show that $\rho_k\rightarrow 0$ and $\sum_k \rho_k = +\infty$. Setting $t=f^{-1}(1/x)$, we see that
\[h'(x) = \frac{1}{(xf^{-1}(1/x))^2}\cdot \frac{1}{f'(f^{-1}(1/x))}= \bigg(\frac{f(t)}{t}\bigg)^2 \cdot \frac{1}{f'(t)} .\]
By Lemma 2.4, we have
\[\bigg(\frac{f(t)}{t}\bigg)^2 \cdot \frac{1}{f'(t)}  = \frac{t^{2s-2}(L(t))^2}{st^{s-1}L(t)+t^sL'(t)} = \frac{t^{s-1}L(t)}{s+tL'(t)/L(t)}\]
Since $t^{s-1}L(t) \overset{t\rightarrow 0}{\longrightarrow} 0$ (see Proposition 1.3.6 of \cite{bingham1989regular}), it suffices to show $tL'(t)/L(t)\overset{t\rightarrow 0}{\longrightarrow} 0$. This is easily seen from Lemma 2.6, as
\[\frac{tL'(t)}{L(t)} = t\cdot\frac{d}{dt}\bigg[\ln\bigg(L(t)\bigg)\bigg] = t\cdot \bigg(-\frac{\varepsilon(t)}{t}\bigg) \rightarrow0\]
Finally, note that
\[\sum_{k=1}^\infty \rho_k \geq \int_N^\infty h'(x) dx = \frac{1}{f^{-1}(1/x)} \Bigg|_{N}^\infty = +\infty,\]
and this completes the proof.
\end{proof}

\end{corollary}

We next state two results that are slight extensions of the arguments found in \cite{bercovici1989analogue}.

\begin{proposition} Let $\rho$ be as defined in Corollary $3.3$, and let $\pi=\{\pi_{k}\}_{k=1}^\infty$ be a non-increasing, real-valued sequence such that $\pi_k\rightarrow 0$ and $\sum_k \pi_k=+\infty$. If $\mathfrak{S}_{\pi}$ is not contained in $\mathfrak{S}_{\rho}$, then
\[\liminf_m f^{-1}\bigg(\frac{1}{m^n}\bigg)\sum_{k=1}^{m^n} \pi_k=0\] 
\begin{proof}

For sufficiently large $m$, we have 
\begin{align*}
\sum_{k=1}^{m^n}\rho_k&\leq \int_{0}^{m^n}h'(x + N - 1)\ dx\\
&\leq \frac{1}{f^{-1}(1/(m^n+N-1))}\\
&\leq \frac{1}{f^{-1}(1/2m^n)}\\
&\leq \frac{d_1}{f^{-1}(1/m^n)},
\end{align*}
where the existence of the positive constant $d_1$ is guaranteed by regular variation of $f^{-1}$ (see Lemma 2.3). Thus
\[\liminf_m f^{-1}\bigg(\frac{1}{m^n}\bigg)\sum_{k=1}^{m^n} \pi_k \leq \liminf_m \frac{\sum_{k=1}^{m^n} \pi_k}{d_1\sum_{k=1}^{m^n}\rho_k}.\]
Thus it suffices to show that the right hand side above equals 0. Suppose to the contrary that $d\sum_{k=1}^{m^n} \pi_k \geq \sum_{k=1}^{m^n}\rho_k$ for some $d>0$. Let T be a positive operator in $\mathfrak{S}_{\pi}\setminus \mathfrak{S}_{\rho}$ with eigenvalues $\{\lambda_k\}_{k=1}^\infty$ listed in nonincreasing order. Then we have
\begin{align*}|T|_\rho=\sum_{k=1}^\infty\lambda_k\rho_k&=\lim_{K\rightarrow\infty}\lambda_K\sum_{\ell=1}^{K}\rho_\ell+\sum_{k=1}^\infty\sum_{\ell=1}^{k}(\lambda_k-\lambda_{k+1})\rho_\ell\\
&\leq d\lim_{K\rightarrow\infty}\lambda_K\sum_{\ell=1}^{K}\pi_\ell+\sum_{k=1}^\infty\sum_{\ell=1}^{k}(\lambda_k-\lambda_{k+1})\pi_\ell\\
&\leq d\sum_{k=1}^\infty\lambda_k\pi_k<+\infty,
\end{align*}
contrary to the fact that $T\notin \mathfrak{S}_{\rho}$.
\end{proof}

\end{proposition}

\begin{corollary} Let $f$ be a gauge function with property $(R_s)$ that satisfies inequality (3), and let $\rho, \pi$ be as in Proposition $3.4$. If $\mathfrak{S}_{\pi}$ is not contained in $\mathfrak{S}_{\rho}$, then $k_\pi(\tau_{\Omega,H_f}) = 0$.

\begin{proof}

This argument mirrors that of Lemma 3.1.  Let $\{\omega_{\ell,m}\}_\ell$ be a partition of $\Omega$ such that $H_{f}(\omega_{\ell,m})\leq D/m^n$. Consider the sequence of projections 
\[P_mv=\sum_{\ell=1}^{m^n} \langle v, \chi_{\omega_{\ell,m}} \rangle \frac{\chi_{\omega_{\ell,m}}}{||\chi_{\omega_{\ell,m}}||_2^2}, \quad v\in L^2(\Omega, H_f). \]
Observe that
\begin{align*}
k_{\pi}(\tau_{\Omega,H_f})\leq|[P_m,\tau_{\Omega,H_f}]|_{\Phi_{\pi}} &\leq  2D\sqrt{n}f^{-1}(1/{m^n})\Phi_\pi(\underset{m^n \text{ times}}{1,1,\dots,1},0,0,\dots)\\
&= 2D\sqrt{n}f^{-1}(1/{m^n})\sum_{k=1}^{m^{n}}\pi_k \rightarrow 0.
\end{align*}
Thus by Lemma 2.8 we have $k_{\pi}(\tau_{\Omega,H_f})=0$.
\end{proof}

\end{corollary}



\begin{theorem}

Let $f:[0,\infty)\rightarrow[0,\infty)$ be function with property $(R_s)$ that satisfies inequality (3), and let $\Omega\subset \mathbb R^n$ be a Borel set so that $0<H_f(\Omega)<+\infty$. Then $\mathfrak{S}_\rho$ is an obstruction ideal for $\tau_{\Omega,H_f}$. Furthermore, if $\Phi\in \mathcal{SN}$, then $k_\Phi(\tau_{\Omega,H_f})\neq 0$ if and only if $\mathfrak{S}_{\Phi} \subset \mathfrak{S}_{\rho}$.
\begin{proof}
We must verify that $\rho$ is regular (as in Definition 2.1) and that the inequalities from the hypotheses of Lemma \ref{lem_obstruction} are satisfied. We proceed by showing that $\sup_{m} m\rho_mf^{-1}(1/m) < +\infty$. Observe that for sufficiently large $m\in\mathbb N$,
\begin{align*}
m\rho_mf^{-1}(1/m)&\leq f^{-1}\Big(\frac{1}{m}\Big)\sum_{k=1}^{m}\rho_k\\
&\leq f^{-1}\Big(\frac{1}{m}\Big)\int_{0}^{m}h'(x+N-1)dx\\
&= f^{-1}\Big(\frac{1}{m}\Big)\Bigg[\frac{1}{f^{-1}(1/(x+N-1))} \Bigg]_{0}^{m}\\
&\leq \frac{f^{-1}(1/m)}{f^{-1}(1/2m)}\longrightarrow 2^{-\frac{1}{s}}
\end{align*}
as $m\rightarrow \infty$. Thus $\sup_mm\rho_mf^{-1}(1/m)<+\infty$.
To verify the remaining inequality, set $t_m=f^{-1}(1/(m+N-1))$ so that
\[m\rho_mf^{-1}(1/m) = \frac{f(t_m)t_{m-N+1}}{f(t_{m-N+1})t_{m}}\cdot\frac{f(t_m)}{t_mf'(t_m)}\geq\frac{f(t_m)t_{2m}}{f(t_{2m})t_{m}}\cdot\frac{f(t_m)}{t_mf'(t_m)}.\]
Regular variation of both $f$ and $f^{-1}$ guarantees that $f(t_m)t_{2m}/f(t_{2m})t_m$ tends to a positive constant, so it suffices to show $\inf_m f(t_m)/t_mf'(t_m) > 0$. Indeed, note that 
\begin{align*}
\frac{f(t_m)}{t_mf'(t_m)}&= \frac{t_m^sL(t_m)}{t_m(st_m^{s-1}L(t_m)+t_m^sL'(t_m))}\\
&= \frac{1}{s+(t_mL'(t_m)/L(t_m))}.
\end{align*}
As shown in the proof of Corollary 3.3, the right-hand side tends to $\frac{1}{s}$ as $m\rightarrow\infty$, and we conclude that $\inf_mm\rho_mf^{-1}(1/m)>0$.

We now verify that $\rho$ is regular. Let $\beta>0$ be chosen so that for every positive integer $m > N$, we have $h(m)<\beta mh'(m)$. The existence of such a $\beta$ is guaranteed by the previous argument. So 
\[\sum_{k=1}^{m}\rho_{k}\leq h(m+N-1) \leq \beta(m+N-1)h'(m+N-1)\leq2\beta m\rho_m,\]
and thus $\rho$ is seen to be regular.

As all of the hypotheses of Lemma \ref{lem_obstruction} are satisfied, we conclude that $0<k_\pi(\tau_{\Omega,H_f})$ and hence $\mathfrak{S}_\rho$ is an obstruction ideal for $\tau_{\Omega,H_f}$.

For the second claim, suppose that $\mathfrak{S}_{\Phi} \subset \mathfrak{S}_{\rho}$. Since $k_\rho(\tau_{\Omega,H_f}) \neq 0$, all ideals smaller are obstruction ideals and hence $k_\Phi(\tau_{\Omega,H_f}) \neq 0$. To complete the proof, assume that $\mathfrak{S}_{\Phi}$ is not contained in $\mathfrak{S}_{\rho}$. From \cite[Proposition 2.6]{voiculescu1990existence}, we know that there exists a non-increasing, real-valued sequence such that $\pi=\{\pi_{k}\}_{k=1}^\infty$ such that $\pi_k\rightarrow 0$ and $\sum_k \pi_k=+\infty$ with the property that $\mathfrak{S}_{\Phi}\subset \mathfrak{S}_{\pi}$. Thus $\mathfrak{S}_{\pi}$ is not contained in $\mathfrak{S}_{\rho}$, and hence $k_\Phi(\tau_{\Omega,H_f})\leq k_\pi(\tau_{\Omega,H_f}) = 0$.
\end{proof}

\end{theorem}

\begin{example}

Suppose that $f(x) = x^s/\log(e/x)$ with $s>1$ for $x\in(0,1]$ and let $C_f$ be the symmetric generalized Cantor set defined in Section 2.3. Let $W(x)$ be defined as the compositional inverse of the real valued function $xe^x$ and
\[\rho=\Bigg\{\frac{\exp\Big(\frac{1}{s} W(s e^s k) - 1\Big)W(se^sk)}{s \, k \, \big(1 + W(s e^s k)\big)}\Bigg\}_{k=1}^{\infty}.\]
Then $\mathfrak{S}_{\rho}$ is the maximum obstruction ideal for the operator $\tau_{C_f,H_f}$.

\end{example}

\begin{proof}

We first show that
\[
f^{-1}(x) = \exp\bigg(1 - \frac{1}{s} W\Big(\frac{s e^s}{x}\Big)\bigg).
\]
To see this, write $y = f(x)$ and set $t = -\log x$, so that $x = e^{-t}$. Then
\[
y \cdot (1+t) = e^{-s t}.
\]
Letting $u = 1 + t$, the equation becomes
\[
y \, u \, e^{s u} = e^s.
\]
Multiplying both sides by $s$ yields
\[
y \, (s u e^{s u}) = s e^s,
\]
which is in the standard form for $W$:
\[
s u = W\Big(\frac{s e^s}{y}\Big).
\]
Finally, recalling that $u = 1+t = 1-\log x$, we solve for $x$ to obtain
\[
x = \exp\bigg(1 - \frac{1}{s} W\Big(\frac{s e^s}{y}\Big)\bigg).
\]
Since $W'(x) = W(x)/(x(1+W(x))$, we see that
\[h'(x) = \frac{\exp\Big(\frac{1}{s} W(s e^s x) - 1\Big)W(se^sx)}{s \, x \, \big(1 + W(s e^s x)\big)}.\]

Since it is easy to see that $f$ has property $(R_s)$, all that remains to be shown is that $f$ satisfies inequality $(3)$. 
A derivative calculation shows that
\begin{align*}
f'(x)\ &= \frac{x^{s-1}(1+s(1-\log(x)))}{(1-\log(x))^2}\\
f''(x) &= x^{s-2}\bigg(\frac{s^2-s}{1-\log(x)} + \frac{2s-1}{(1-\log(x))^2} + \frac{2}{(1-\log(x))^3}\bigg)
\end{align*}
Letting $L(x) = 1/(1-\log(x))$, we see that:
\begin{align*}
    &\phantom{=}x(2f(x)f'(x)+xf(x)f''(x)-2xf'(x)^2)/f(x)^2\\
    &= 2(s+L(x))+s^2-s+(2s-1)L(x)+2L(x)^2-2(s+L(x))^2\\
    &= s-s^2 + (1-2s)L(x) < 0,
\end{align*}
and thus inequality $(3)$ is satisfied. The claim follows from an application of Theorem 3.6. 
\end{proof}

\section{Ampliation Homogeneity for $k_\rho$}

To begin, we first show a kind of ampliation homogeneity for some quasicentral moduli which are related to $k_\rho$ ($\rho$ as in Section 3). For the remainder of the paper, we assume our gauge functions $f$ have property $(R_{s})$. Recall the notation $h(x)=1/(f^{-1}(1/x))$.

\begin{lemma} 
Let $f:[0,\infty)\rightarrow[0,\infty)$ be a function with property $(R_s)$ that satisfies inequality (3). For a fixed positive integer $m$ and fixed $\varepsilon>0$, let $\rho^{(\varepsilon)} = \{\rho_k^{(\varepsilon)}\}_{k=0}^\infty$ with $\rho_k^{(\varepsilon)}=\rho_{k+N_m}$ and $N_m$ chosen so that for every $k>N_m$ we have
\[\bigg|\frac{f^{-1}(1/k)}{f^{-1}(1/mk)} - m^{1/s}\bigg| < \varepsilon. \]\\
If $X_j$ is a sequence of finite rank operators such that $||X_j|| \rightarrow 0$ and $|X_j|_{\rho^{(\varepsilon)}} < C < +\infty$, then there exists a positive constant $K$ dependent on only $C$ and $\rho$ so that,
\[\lim_{j \rightarrow \infty} \bigg||X_j \tensor I_m |_{\rho^{(\varepsilon)}} - m^{1/s}|X_j|_{\rho^{(\varepsilon)}}\bigg| \leq \varepsilon K.\]
\end{lemma}

\begin{proof}

Let $s_k^{(j)}$ be the singular values of $X_j$ listed in nonincreasing order with the convention that $s_{k}^{(j)} = 0$ for $k > \text{rank}(X_j^*X_j)$. Then

\begin{align*}
|X_j \tensor I_m |_{\rho^{(\varepsilon)}} &= \sum_{k=1}^\infty s_k^{(j)} \sum_{\ell = (k-1)m +1}^{mk} \rho_{\ell}^{(\varepsilon)}\\
&=s_1^{(j)}\sum_{\ell =1}^{m} \rho_{\ell}^{(\varepsilon)}+\sum_{k=2}^\infty s_k^{(j)} \Bigg(\sum_{\ell =1}^{mk} \rho_{\ell}^{(\varepsilon)} - \sum_{\ell = 1}^{(k-1)m} \rho_{\ell}^{(\varepsilon)}\Bigg)\\
&= \lim_{r\rightarrow\infty}s_{r}^{(j)}\sum_{\ell = 1}^{mr}\rho_{\ell}^{(\epsilon)}+ \sum_{k=1}^\infty \big(s_k^{(j)} - s_{k+1}^{(j)}\big) \sum_{\ell = 1}^{mk} \rho_{\ell}^{(\varepsilon)}.\\
&=\sum_{k=1}^\infty \big(s_k^{(j)} - s_{k+1}^{(j)}\big) \sum_{\ell = 1}^{mk} \rho_{\ell}^{(\varepsilon)}.\\
\end{align*}
To succinctly handle the cumbersome notation introduced by the two (left) truncations of the sequence $\{h'(k)\}_{k=1}^\infty$, we introduce $\gamma_m:= N+N_m - 1$. Next, suppose that $k>N_m+N-1$ so that $h'(mk) < \varepsilon/m\gamma_m$. We see that
\begin{align*}\bigg|h(mk+\gamma_m) - m^{1/s}h(k+\gamma_m)\bigg| &< \bigg|h(mk+\gamma_m) - h(m(k+\gamma_m))\bigg|\\
&+\bigg|h(m(k+\gamma_m)) - m^{1/s}h(k+\gamma_m)\bigg|.
\end{align*}
As established in Lemma 3.2, $h'(x)$ is a nonincreasing function for $x>N$, so the Mean Value Theorem implies that\\
\[\bigg|h(mk+\gamma_m) - h(m(k+\gamma_m))\bigg|\leq (m\gamma_m - \gamma_m) h'(mk+\gamma_m)\leq \varepsilon.\]
For the second summand, the hypotheses of the lemma imply that
\begin{align*}
\bigg|h(m(k+\gamma_m)) - m^{1/s}h(k+\gamma_m)\bigg| &= h(k+\gamma_m)\bigg|\frac{h(m(k+\gamma_m))}{h(k+\gamma_m)} - m^{1/s}\bigg|\\
&\leq \varepsilon h(k+\gamma_m)
\end{align*}
There exists a positive constant $\beta$ so that $k\rho_{k}^{(\varepsilon)}f^{-1}(1/(k+\gamma_m))>\beta/2$ (see the proof of Theorem 3.6). Thus
\[\bigg|h(mk+\gamma_m) - m^{1/s}h(k+\gamma_m)\bigg| < \varepsilon+\varepsilon h(k+\gamma_m) < \varepsilon \beta k{\rho_k^{(\varepsilon)}}.\]
The following lines of inequalities bound $\sum_{\ell = 1}^{mk} \rho_{\ell}^{(\varepsilon)}$ between two expressions of the form $\alpha(k)+ m^{1/s}\sum_{\ell = 1}^{k}\rho_\ell^{(\varepsilon)}$ where $\alpha(k)$ is solely dependent on $k$ and satisfies $\lim_{j\rightarrow \infty}|\sum_{k=1}^\infty \big(s_k^{(j)} - s_{k+1}^{(j)}\big)\alpha(k)| \leq \varepsilon K$. For the first inequality, observe that 
\begin{align*}
\sum_{\ell = 1}^{mk} \rho_{\ell}^{(\varepsilon)} &\leq \int_{0}^{mk} h'(x+\gamma_m)\ dx\\
&= h(mk+\gamma_m)-h(\gamma_m)\\
&\leq \bigg( \varepsilon \beta k{\rho_k^{(\varepsilon)}}-h(\gamma_m) \bigg)+ m^{1/s}h(k+\gamma_m)\\
&= \bigg(\varepsilon \beta k{\rho_k^{(\varepsilon)}} +(m^{1/s}-1)h(\gamma_m) \bigg)+ m^{1/s}\int_{0}^{k}h'(x+\gamma_m)\ dx\\
&\leq \bigg({m^{1/s}\rho_0^{(\varepsilon)}}+\varepsilon \beta k{\rho_k^{(\varepsilon)}} +(m^{1/s}-1)h(\gamma_m) \bigg)+ m^{1/s}\sum_{\ell = 1}^{k}\rho_\ell^{(\varepsilon)}.
\end{align*}
Likewise,
\begin{align*}
\sum_{\ell = 1}^{mk} \rho_{\ell}^{(\varepsilon)} &\geq  -\rho_0^{(\varepsilon)}+\int_{0}^{mk} h'(x+\gamma_m)\ dx\\
&=-\rho_0^{(\varepsilon)}+ h(mk+\gamma_m)-h(\gamma_m)\\
&\geq - \bigg(\rho_0^{(\varepsilon)}+ \varepsilon \beta k{\rho_k^{(\varepsilon)}}+h(\gamma_m) \bigg)+ m^{1/s}h(k+\gamma_m)\\
&= -\bigg(\rho_0^{(\varepsilon)}+\varepsilon \beta k{\rho_k^{(\varepsilon)}} -(m^{1/s}-1)h(\gamma_m) \bigg)+ m^{1/s}\int_{0}^{k}h'(x+\gamma_m)\ dx\\
&\geq -\bigg({m^{1/s}\rho_0^{(\varepsilon)}}+\varepsilon \beta k{\rho_k^{(\varepsilon)}} -(m^{1/s}-1)h(\gamma_m) \bigg)+ m^{1/s}\sum_{\ell = 1}^{k}\rho_\ell^{(\varepsilon)}.
\end{align*}
Notice that the inequality ${m^{1/s}\rho_0^{(\varepsilon)}} \leq \rho_0^{(\varepsilon)}$ is used in the last line to provide symmetry with the previous case. With the notation
\[\alpha_{\pm}(k):={m^{1/s}\rho_0^{(\varepsilon)}}+\varepsilon \beta k{\rho_k^{(\varepsilon)}} \pm(m^{1/s}-1)h(\gamma_m),\]
we obtain
\[\sum_{k=1}^\infty \big(s_k^{(j)} - s_{k+1}^{(j)}\big)\bigg(-\alpha_{-}(k)+ m^{1/s}\sum_{\ell = 1}^{k}\rho_\ell^{(\varepsilon)}\bigg)\leq\sum_{k=1}^\infty \big(s_k^{(j)} - s_{k+1}^{(j)}\big) \sum_{\ell = 1}^{mk} \rho_{\ell}^{(\varepsilon)}\]
\[\leq\sum_{k=1}^\infty \big(s_k^{(j)} - s_{k+1}^{(j)}\big)\bigg(\alpha_{+}(k)+ m^{1/s}\sum_{\ell = 1}^{k}\rho_\ell^{(\varepsilon)}\bigg).\]\\
We claim that
\[\lim_{j\rightarrow \infty}\bigg|\sum_{k=1}^\infty \big(s_k^{(j)} - s_{k+1}^{(j)}\big)\alpha_{\pm}(k)\bigg| \leq 2\beta\varepsilon C.\]
Indeed, for any constant $D$, we have
\[\lim_{j\rightarrow \infty}\sum_{k=1}^\infty \big(s_k^{(j)} - s_{k+1}^{(j)}\big)D = \lim_{j\rightarrow\infty}s_1^{(j)}D = 0,\]
so it remains to show that 
\[\lim_{j\rightarrow \infty}\bigg|\sum_{k=1}^\infty \big(s_k^{(j)} - s_{k+1}^{(j)}\big)k\rho_k^{(\varepsilon)}\bigg| \leq 2C.\]
For every positive integer $q$, observe that
\[\sum_{k=q+1}^\infty s_k\rho_k^{(\varepsilon)}\leq\sum_{k=q}^\infty s_{k+1}\rho_{k}^{(\varepsilon)}\leq\sum_{k=q}^\infty s_k\rho_k^{(\varepsilon)},\]
which further implies
\[0\leq \sum_{k=q}^\infty s_{k+1}(\rho_k^{(\varepsilon)}-\rho_{k+1}^{{(\varepsilon)}})\leq s_q\rho_q^{(\varepsilon)}.\]
Thus
\begin{align*}
\lim_{j\rightarrow \infty}\bigg|\sum_{k=1}^\infty \big(s_k^{(j)} - s_{k+1}^{(j)}\big)k\rho_k^{(\varepsilon)}\bigg| &= \lim_{j\rightarrow \infty}\bigg|\sum_{k=1}^\infty s_k^{(j)}\rho_k^{(\varepsilon)} - \sum_{k=1}^\infty\sum_{\ell=k}^{\infty}s_{\ell+1}^{(j)}\big(\rho_{\ell}^{(\varepsilon)}-\rho_{\ell+1}^{(\varepsilon)}\big)\bigg|\\
&\leq \lim_{j\rightarrow \infty}\bigg|\sum_{k=1}^\infty s_k^{(j)}\rho_k^{(\varepsilon)}\bigg| + \bigg|\sum_{k=1}^\infty\sum_{\ell=k}^{\infty}s_{\ell+1}^{(j)}\big(\rho_{\ell}^{(\varepsilon)}-\rho_{\ell+1}^{(\varepsilon)}\big)\bigg|\\
&\leq \lim_{j\rightarrow \infty}\bigg|\sum_{k=1}^\infty s_k^{(j)}\rho_k^{(\varepsilon)}\bigg| + \bigg|\sum_{k=1}^\infty s_{k}^{(j)}\rho_{k}^{(\varepsilon)}\bigg|\\
&\leq 2C
\end{align*}
We therefore conclude that
\[\lim_{j\rightarrow \infty}\bigg|  \sum_{k=1}^\infty \big(s_k^{(j)} - s_{k+1}^{(j)}\big) \sum_{\ell = 1}^{mk} \rho_{\ell}^{(\varepsilon)} - m^{1/s}\sum_{k=1}^\infty \big(s_k^{(j)} - s_{k+1}^{(j)}\big)\sum_{\ell = 1}^{k}\rho_\ell^{(\varepsilon)} \bigg|\leq 2\beta\varepsilon C,\]\\
and since,
\[m^{1/s}\sum_{k=1}^\infty \big(s_k^{(j)} - s_{k+1}^{(j)}\big)\sum_{\ell = 1}^{k}\rho_\ell^{(\varepsilon)} = m^{1/s}\sum_{k=1}^{\infty}s_k\rho_k = m^{1/s}|X_j|_{\rho^{(\varepsilon)}},\]
the lemma follows.
\end{proof}

\begin{corollary}

Let $f,\varepsilon,\rho^{(\varepsilon)}$ be defined as above. If $X_j$ is a sequence of $n$-tuples of finite rank operators such that $|X_j|_{\rho^{(\varepsilon)}}<C<+\infty$ and $||X_j||\rightarrow 0$, then there exists a positive constant $K$ dependent only on $C$ and $\rho$ so that,
\[\lim_{j \rightarrow \infty} \bigg||X_j \tensor I_m |_{\rho^{(\varepsilon)}} - m^{1/s}|X_j|_{\rho^{(\varepsilon)}}\bigg| \leq \varepsilon K.\]
\end{corollary}

The next lemma shows that, even though the ampliation result from the previous lemma only applies to norms of the form $|\cdot|_{\rho^{(\varepsilon)}}$, it can still be used to calculate $k_\rho(\tau)$. Let $S\pi$ denote the left shift of a sequence $\pi$, that is, $S\pi=\{\xi_{k}\}_{k=1}^{\infty}$ with $\xi_k=\pi_{k+1}$.

\begin{lemma}

 Let $\tau$ be an $n$-tuple of commuting self-adjoint operators, and let $\pi$ be a regular sequence. Then $k_{\pi}(\tau)=k_{S\pi}{(\tau)}$.

\end{lemma}

\begin{proof}

Begin by noting that for any operator $T$ with $|T|_\pi< +\infty$, we have $|T|_\pi\geq |T|_{S\pi}$, so clearly $k_\pi(\tau)\geq k_{S\pi}(\tau)$. For the other inequality, let $A_j$ be a sequence of positive finite rank contractions so that $\lim_j |[A_j, \tau]|_{S\pi} = k_{S\pi}(\tau) \text{ and} \lim_j ||[A_j, \tau]|| = 0.$ If $s_k^{j}$ are the singular values of $[A_j, \tau]$ listed in nonincreasing order, then
\[k_\pi(\tau) \leq \lim_j |[A_j, \tau]|_{\pi} = \lim_j \sum_{k=1}^\infty s_k^{(j)}\pi_k = \lim_j \bigg(s_1^{(j)}\pi_1+\sum_{k=2}^\infty s_k^{(j)}\pi_k \bigg)\]
\[\leq \lim_j \bigg(s_1^{(j)}\pi_1+\sum_{k=1}^\infty s_k^{(j)}\pi_{k+1} \bigg)= \lim_j||[A_j,\tau]||\pi_1 + \lim_j \sum_{k=1}^\infty s_k^{(j)}\pi_{k+1} =k_{S\pi}(\tau).\]
Thus $k_{\pi}(\tau)=k_{S\pi}(\tau)$, as desired.
\end{proof}

Since $\rho^{(\varepsilon)}$ is the sequence $\rho$ left-shifted a finite number of times, we deduce the following result.

\begin{corollary}

Let $\tau$ be an n-tuple of operators such that $0<k_\rho(\tau)<+\infty$. Then $k_{\rho}(\tau)=k_{\rho^{(\varepsilon)}}(\tau)$.

\end{corollary}

We now state and prove an extension of the homogeneity result in \cite[Theorem 3.1]{voiculescu2021formula}.

\begin{theorem} 

Let $f:[0,\infty)\rightarrow[0,\infty)$ be a function with property (R$_{s}$) for some $s> 1$ that satisfies inequality $(3)$. Then

\[k_\rho(\tau_{\Omega,H_f}\tensor I_m) = m^{1/s}k_{\rho}(\tau_{\Omega,H_f}).\]

\end{theorem}

\begin{proof}

Fix $\varepsilon > 0$. By Lemma 3.3 of \cite{voiculescu2021formula}, we can choose a sequence $A_j \subset R_1^+$ such that $A_j\uparrow I$ with 

\[k_{\rho^{(\varepsilon)}}(\tau_{\Omega,H_f}\tensor I_m) = \lim_{j\rightarrow\infty} |[A_j,\tau_{\Omega,H_f}]\tensor I_m|_{\rho^{(\varepsilon)}}\]
and
\[\lim_{j\rightarrow\infty} ||[A_j,\tau_{\Omega,H_f}]||=0.\]
For a constant $K>0$, Corollary 4.2 implies that 
\begin{align*}
k_{\rho^{(\varepsilon)}}(\tau_{\Omega,H_f}\tensor I_m)&\geq\lim_{j\rightarrow\infty} m^{1/s}|[A_j,\tau_{\Omega, H_f}]|_{\rho^{(\varepsilon)}} - K\varepsilon\geq m^{1/s}k_{\rho^{(\varepsilon)}}(\tau_{\Omega,H_f})-K\varepsilon,
\end{align*}
and thus
\begin{align*}
k_\rho(\tau_{\Omega,H_f}\tensor I_m) &= k_{\rho^{(\varepsilon)}}(\tau_{\Omega,H_f}\tensor I_m)\\
&\geq m^{1/s}k_{\rho^{(\varepsilon)}}(\tau_{\Omega,H_f})- K\varepsilon\\
&= m^{1/s}k_{\rho}(\tau_{\Omega,H_f})-K\varepsilon.
\end{align*}
Next, choose a sequence $B_j \subset R_1^+$ such that $B_j\uparrow I$ with 
\[k_{\rho^{(\varepsilon)}}(\tau_{\Omega,H_f}) = \lim_{j\rightarrow\infty} |[B_j,\tau_{\Omega,H_f}]|_{\rho^{(\varepsilon)}},\]
\noindent and
\[\lim_{j\rightarrow\infty} ||[B_j,\tau_{\Omega,H_f}]||=0.\]
Using Corollary 4.2 once more, we see that
\[m^{1/s}k_{\rho^{(\varepsilon)}}(\tau_{\Omega,H_f})\geq\lim_{j\rightarrow\infty} |[B_j,\tau_{\Omega,H_f}]\tensor I_m|_{\rho^{(\varepsilon)}}-K\varepsilon \geq k_{\rho^{(\varepsilon)}}(\tau_{\Omega,H_f} \tensor I_m)-K\varepsilon.\]
Thus
\begin{align*}
m^{1/s}k_\rho(\tau_{\Omega,H_f}) &= m^{1/s}k_{\rho^{(\varepsilon)}}(\tau_{\Omega,H_f})\\
&\geq k_{\rho^{(\varepsilon)}}(\tau_{\Omega,H_f} \tensor I_m)-K\varepsilon\\
&= k_{\rho}(\tau_{\Omega,H_f}\tensor I_m)-K\varepsilon.
\end{align*}
Since $\varepsilon$ is arbitrarily small, the result follows.
\end{proof}
    
\section{Exact Formula for the Quasicentral Modulus for Generalized Hausdorff Measures}

For this section we once again assume $f:[0,\infty)\rightarrow[0,\infty)$ a function with property $(R_s)$ for some $s>1$ that satisfies inequality (3). Let $n=\lfloor s\rfloor$ +1 and recall that $\rho:=\{h'(k+N-1)\}_{k=1}^\infty$ with $N$ being chosen so that $\rho$ defines a symmetric-norming function (see Corollary 3.3). If $\tau$ is an $n$-tuple of commuting self-adjoint operators with $\sigma(\tau)\subset C_f$, and $m$ is the multiplicity function for $\tau$, we show that 
${(k_\rho(\tau))}^s$ is proportional to $\int_{\sigma(\tau)} m(x) dH_f(x)$. To begin, we first show that $k_\rho(\tau) = 0$ when ${\mathcal{H}^{f}_{ac}} = \{0\}$.

\begin{lemma}

If $f$ is a gauge function with property $(R_s)$ that satisfies inequality (3), $\Omega\subset C_f$ is a closed Borel set, and $\tau_{\Omega,\nu}$ is multiplication by the variables on $L^2(\Omega,\nu)$ for some Borel probability measure $\nu$, then there exists a constant $D>0$ dependent only on $n$ such that
\[k_\rho(\tau_{\Omega,\nu}) \leq D (H_{f}(\Omega))^{1/s}.\]
\end{lemma}

\begin{proof}

This argument mirrors that found in Lemma 5.1 of \cite{voiculescu2021formula}. Let $W_L=\{w : |w| = L \text{ and } C_f^w \cap \Omega \neq \varnothing \}$ and $U_L = \bigcup_{w\in W_L}C_f^w$. For any $\varepsilon>0$, choose an $L_0$ sufficiently large so that $N<|W_{L_0}|$ and $H_f(\Omega) > H_f(U_{L_0})-\varepsilon$. Let 
\[P_Lv=\sum_{w\in W_L} \langle v, \chi_{C_f^w} \rangle \frac{\chi_{C_f^w}}{||\chi_{C_f^w}||_2^2},\quad v\in L^2(\Omega, \nu).\]
Recall from Section 2.3 that is constructed so that $\text{diam}(C_f^w)\leq \sqrt{n}f^{-1}(1/2^{nL})$. If $P_w = \langle \cdot, \chi_{C_f^w} \rangle \frac{\chi_{C_f^w}}{||\chi_{C_f^w}||_2^2}$, observe that for $L>L_0$
\[||[P_L,\tau_{\Omega,\nu}]|| = \max_{C_f^w}||[P_w,\tau_{\Omega,\nu}]||\leq\max_{C_f^w}2\text{diam}(C_f^w)\leq 2\sqrt{n} f^{-1}(1/2^{nL}).\]
Since $P_L\uparrow I$ as $L\rightarrow\infty$, we have
\begin{align*}|[P_L,\tau_{\Omega,\nu}]|_\rho&\leq ||[P_L,\tau_{\Omega,\nu}]|| \sum_{k=1}^{\text{rank}(P_L)}\rho_k\\
&\leq||[P_L,\tau_{\Omega,\nu}]|| \int_{1}^{|W_L|}h'(x+N-1)\ dx\\
&\leq\frac{ 2\sqrt{n} f^{-1}(1/2^{nL})}{f^{-1}(1/2|W_L|)}.
\end{align*}
As seen in Section 2.3, there exists a constant $\xi>0$ dependent only on $n$ such that $\xi\cdot2^{-nL}\leq H_f(C_f^w)$. If $w\in W_L$ so that $L>L_0$, then $|W_L| H_f(C_f^w)=H_f(U_L) <H_{f}(\Omega)+\varepsilon$.  Let $t_L:=H_f(C_f^w)/\xi$. Then
\[\frac{1}{f^{-1}(1/2|W_L|)} < \frac{1}{f^{-1}(\xi t_L/(2H_{f}(\Omega)+2\varepsilon))}.\]
In the next chain of inequalities, we will be using regular variation of $f^{-1}$ once more. To make the usage more apparent, we set $\alpha := \xi/(2H_{f}(\Omega)+2\varepsilon)$. With $L$ chosen large enough so that ${f^{-1}(t_L)/f^{-1}(\alpha t_L)} < \alpha^{-1/s}+\varepsilon$, we have
\begin{align*}
|[P_L,\tau_{\Omega,\nu}]|_\rho &\leq \frac{ 2\sqrt{n}f^{-1}(1/2^{nL})}{f^{-1}(1/2|W_L|)}\\
&\leq \frac{ 2\sqrt{n}f^{-1}(t_L)}{f^{-1}(\alpha t_L)}\\
&< 2\sqrt{n}(\alpha^{-1/s}+\varepsilon)\\
&< \frac{2\sqrt{n}}{\xi^{1/s}} (2(H_{f}(\Omega)+\varepsilon)^{1/s}+2\varepsilon)\\
&\leq D(H_f(\Omega))^{1/s} + 4\varepsilon^{1/s}
\end{align*}
Since $\varepsilon$ can be chosen arbitrarily small, we conclude $k_\rho(\tau_{\Omega,\nu}) \leq D (H_{f}(\Omega))^{1/s}.\qedhere$
\end{proof}

\begin{corollary}

Let $f$ be a gauge function with property $(R_s)$ that satisfies inequality (3), and let $\tau$ be an $n$-tuple of commuting self-adjoint operators with $\sigma(\tau)\subset C_f$. If $E_{\tau}$ is singular with respect to $H_f$, then $k_\rho(\tau)=0$.

\end{corollary}

\begin{proof}
Observe that we may assume $\tau$ has a cyclic vector $\eta$ and hence $\tau$ is unitarily equivalent to some $\tau_{\Omega,\nu}$ defined on $L^2(\Omega,\nu)$ with $\Omega\subset C_f$ and $\nu$ a Borel probability measure singular with respect to $H_f$. 

Given $\varepsilon > 0$, we may choose a set $\omega_m\subset\Omega$ which is a union of generation sets of the same size so that $H_f(\omega_m) < \varepsilon$ and $\nu(\Omega\setminus\omega_m) < 2^{-m}$ for all $x\in \Hilb_{ac}$. By the previous lemma, we see that 
\[k_\rho(\tau E_\tau(\omega_m))<D \cdot (H_f(\omega_m))^{1/s} < D\cdot \varepsilon^{1/s}\]
for some constant $D>0$ dependent only on $f$. Since $E_\tau(\omega_m)$ is in the commutant $(\{\tau\})'$, and since $\eta$ is cyclic for $\tau$, the inequality $||\eta - E_\tau(\omega_m)\eta||^2 = \nu(\Omega\setminus\omega_m) \leq 2^{-m}$ implies that $E_\tau(\omega_m)$ converges strongly to $I$. Choosing $A_k\in \mathcal{R}_1^{+}$ so that $A_k\leq E_\tau(\omega_k)$, $A_k\uparrow I$, and $k_\rho(\tau E_\tau(\omega_n)) \geq \lim_{k\rightarrow\infty }|[A_k,\tau E_\tau(\omega_m)]|_\rho$ yields
\[k_\rho(\tau)\leq \lim_{k\rightarrow\infty }|[A_k,\tau ]|_\rho \leq k_\rho(\tau E_{\tau}(\omega_m))<D\cdot\varepsilon.\]
letting $\varepsilon$ tend to 0 yields the result.
\end{proof}

\begin{theorem}

Let $f, \tau$ be as above. If $m$ is the multiplicity function for $\tau$, then 
\[(k_\rho(\tau))^s = \kappa_s^{(f)} \int_{\sigma(\tau)} m(x)dH_f(x),\]
where $\kappa_s^{(f)} = (k_{\rho}(\tau_{C_f}))^s/H_{f}(C_f).$
    
\end{theorem}

\begin{proof}

Begin by noting that $0<\kappa_s^{(f)}<\infty$ by Theorem 3.6. By Corollary 5.2, we may assume $\Hilb = \Hilb_{ac}$. Thus $\tau$ is unitarily equivalent to $\bigoplus_{j=1}^\infty\tau_{\Omega_j}$ for some Borel sets $\Omega_j$. Furthermore, Lemma 2.10 shows that we can assume $\tau = \bigoplus_{j=1}^N \tau_{\Omega_j}$. 


Let $W_{L_j}=\{w : |w| = L \text{ and } C_f^w \cap \Omega_j \neq \varnothing \}$. First, assume that $\Omega_j=\bigcup_{w\in W_{L_j}} C_f^w$ for some fixed positive integer $L_j$. Since each generation set is a union of smaller generation sets, we may assume $L_j = L$ for some fixed $L$. Let $M = \sum_{j=1}^{N}|W_{L_j}|$. Thus for any generation set $C_f^w$ with $|w|=L$ 
\begin{align*}(k_\rho(\tau))^s &= \bigg(k_{\rho}\bigg(\bigoplus_{j=1}^N \tau_{\Omega_j}\bigg)\bigg)^s = (k_{\rho}(I_M \tensor \tau_{C_f^w}))^s = M(k_{\rho}(\tau_{C_f^w}))^s.\\
\end{align*}
Given this, it suffices to show
\[(k_\rho(\tau_{C_f^w}))^s=\kappa_s^{(f)}H_f(C_f^w)\]
for any choice of $w$. Observe that
\[k_{\rho}(\tau_{C_f})=k_\rho(\tau_{C_f^w} \tensor I_{2^{nL}})= 2^{nL/s}k_\rho(\tau_{C_f^w})\]
so we have
\[(k_\rho(\tau_{C_f^w}))^s = (2^{-nL/s}k_\rho(\tau_{C_f}))^s = 2^{-nL}\kappa_{s}^{(f)}H_f(C_f)=\kappa_{s}^{(f)}H_f(C_f^w).\]
Thus result is proven in this case.

Next, suppose that $\Omega = \bigcup_{j=1}^N U_j$ for some open sets $\{U_j\}_{j=1}^N$. Note that each $U_j$ can be well approximated by a finite union of some $C_f^w$, and these sets can be chosen so that for every $w$, $|w|=L_j$. Choosing $L=\min\{L_j\}$ and $w$ with $|w|=L$, we have reduced the problem to showing that
\[(k_\rho(\tau_{C_f^{w}}\tensor I_{N2^{nL}}))^s=N2^{nL}\kappa_s^{(f)}H_f(C_f^{w})\]
which is true by the previous argument.

    Finally, let $\Omega = \bigcup_{j=1}^N \Omega_j$ for $\Omega_j$ arbitrary Borel sets. Since for each $\Omega_j$ there is $K_j$ compact and $U_j$ open so that $K_j\subset \Omega_j \subset U_j$ and $H_f(U_j\setminus K_j)<\varepsilon$ for some $\varepsilon>0$. Since $|k_\rho(\tau_{\Omega_j}) - k_\rho(\tau_{U_j})| \leq k_\rho(\tau_{U_j\setminus K_j})\leq \varepsilon$, we can reduce the argument to the previous case and conclude the proof.

\end{proof}

\newpage
\nocite{*}

\bibliographystyle{plain}
\bibliography{perturbationreferences}

@article{voiculescu1979some,
  title={Some results on norm-ideal perturbations of Hilbert space operators},
  author={Voiculescu, D.},
  journal={J. Oper. Theory},
  pages={3--37},
  year={1979},
  publisher={JSTOR}
}

@article{david1990s,
  title={$s$-Numbers of singular integrals for the invariance of absolutely continuous spectra in fractional dimensions},
  author={David, G. and Voiculescu, D.},
  journal={J. Funct. Anal.},
  volume={94},
  number={1},
  pages={14--26},
  year={1990},
  publisher={Academic Press}
}

@article{voiculescu1981some,
  title={Some results on norm-ideal perturbations of Hilbert space operators. {II}},
  author={Voiculescu, D.},
  journal={J. Oper. Theory},
  pages={77--100},
  year={1981},
  publisher={JSTOR}
}

@article{arveson1977notes,
  title={Notes on extensions of ${C}$*-algebras},
  author={Arveson, W.},
  journal={Duke Math. J.},
  volume={44},
  number={2},
  pages={329},
  year={1977},
  publisher={Duke University Press}
}

@inproceedings{bercovici1989analogue,
  title={The Analogue of {K}uroda’s Theorm for n-tuples},
  author={Bercovici, H. and Voiculescu, D.},
  booktitle={The Gohberg Anniversary Collection: Volume I: The Calgary Conference and Matrix Theory Papers and Volume II: Topics in Analysis and Operator Theory},
  pages={549--552},
  year={1989},
  organization={Springer}
}

@article{voiculescu1990existence,
  title={On the existence of quasicentral approximate units relative to normed ideals. Part {I}},
  author={Voiculescu, D.},
  journal={J. Funct. Anal.},
  volume={91},
  number={1},
  pages={1--36},
  year={1990},
  publisher={Academic Press}
}

@article{voiculescu2021formula,
  title={The formula for the quasicentral modulus in the case of spectral measures on fractals},
  author={Voiculescu, D.},
  journal={J. Fract. Geom.},
  volume={8},
  number={4},
  pages={347--361},
  year={2021}
}

@incollection{besicovitch2019complementary,
  title={On the complementary intervals of a linear closed set of zero {L}ebesgue measure},
  author={Taylor, S. J.},
  booktitle={Classics On Fractals},
  pages={268--282},
  year={2019},
  publisher={CRC Press}
}

@article{hatano2004hausdorff,
  title={Hausdorff measures and packing premeasures},
  author={Hatano, K.},
  journal={Mem. Fac. Sci. Eng., Shimane Univ.},
  volume={37},
  pages={5--14},
  year={2004},
  publisher={Shimane University}
}

@article{ikeda2023quasicentral,
  title={Quasicentral Modulus and Self-similar Sets: A Supplementary Result to {V}oiculescu’s Work},
  author={Ikeda, K. and Izumi, M.},
  journal={Integral Equ. Oper. Theory},
  volume={95},
  number={2},
  pages={13},
  year={2023},
  publisher={Springer}
}

@Book{GohK69,
  author =       "I. C. Gohberg and M. G. Krein",
  year =         "1969",
  title =        "Introduction to the Theory of Linear Non-Self-Adjoint
                 Operators",
  publisher =    "AMS",
  address =      "Providence, RI, USA",
}

@article{von1909beschrankte,
  title={{\"U}ber beschr{\"a}nkte quadratische Formen, deren Differenz vollstetig ist},
  author={Weyl, H.},
  journal={Rend. Circ. Mat. Palermo},
  volume={27},
  number={1},
  pages={373--392},
  year={1909},
  publisher={Springer Milan Milan}
}

@article{von1935charakterisierung,
  title={Charakterisierung des Spektrums eines Integral operators},
  author={von Neumann, J.},
  journal={Actualit{\'e}s Sci. Indust.},
  volume={229},
  pages={l--20},
  year={1935}
}

@article{berg1971extension,
  title={An extension of the {W}eyl-von {N}eumann theorem to normal operators},
  author={Berg, I. D.},
  journal={Trans. of the Amer. Math. Soc.},
  volume={160},
  pages={365--371},
  year={1971}
}

@article{kuroda1958theorem,
  title={On a theorem of {W}eyl-von {N}eumann},
  author={Kuroda, S. T.},
  journal={Proc. Japan Acad.},
  volume={34},
  number={1},
  pages={11--15},
  year={1958},
  publisher={The Japan Academy}
}

@article{kato1957perturbation,
  title={Perturbation of continuous spectra by trace class operators},
  author={Kato, T.},
  journal={Proc. Japan Acad.},
  volume={33},
  number={5},
  pages={260--264},
  year={1957},
  publisher={The Japan Academy}
}

@article{rosenblum1957perturbation,
  title={Perturbation of the continuous spectrum and unitary equivalence.},
  author={Rosenblum, M.},
  journal={Pacific J. Math.},
  volume={7},
  number={4},
  pages={997--1010},
  year={1957}
}

@article{carey1974perturbation,
  title={Perturbation by trace class operators},
  author={Carey, R. W. and Pincus, J. D.},
  journal={Bull. Amer. Math. Soc.},
  volume={80},
  number={4},
  pages={758--759},
  year={1974}
}

@book{bingham1989regular,
  title={Regular variation},
  author={Bingham, N. H. and Goldie, C. M. and Teugels, J. L.},
  number={27},
  year={1989},
  publisher={Cambridge University Press}
}

@article{voiculescu1976non,
  title={A non-commutative Weyl-von Neumann theorem},
  author={Voiculescu, D.},
  journal={Rev. roum. math. pures et appl.},
  volume={21},
  pages={97--113},
  year={1976}
}

\end{document}